\documentclass[12pt,reqno]{amsart}
\usepackage{amssymb, amsmath, amsthm}
\usepackage{enumerate}
\newtheorem{thm}{Theorem}[section]
\newtheorem{lem}[thm]{Lemma}

\newtheorem{prop}[thm]{Proposition}

\numberwithin{equation}{section}

\newcommand{\bel}{\begin{equation} \label}
\newcommand{\ee}{\end{equation}}
\def\beq{\begin{equation}}
\def\eeq{\end{equation}}
\newcommand{\bea}{\begin{eqnarray}}
\newcommand{\eea}{\end{eqnarray}}
\newcommand{\beas}{\begin{eqnarray*}}
\newcommand{\eeas}{\end{eqnarray*}}
\newcommand{\pd}{\partial}
\newcommand{\mdiv}[1]{\ensuremath{\mathrm{div} \left( #1 \right)}}
\newcommand{\dd}{\mbox{d}}

\newcommand{\R}{\mathbb{R}}
\newcommand{\C}{\mathbb{C}} 
\newcommand{\N}{\mathbb{N}}

\newcommand{\cD}{\mathcal{D}}
\newcommand{\cL}{\mathcal{L}}
\newcommand{\cO}{\mathcal{O}}
\newcommand{\Hin}{\mathcal{H}_{\mathrm{in},T_0}}
\newcommand{\Gi}{S_{\mathrm{in}}}
\newcommand{\Go}{S_{\mathrm{out}}}
\newcommand{\dom}{\mathrm{Dom}}

\renewcommand{\div}{\mathrm{div}\,}  

\newcommand{\supp}{\mathrm{supp}\,}  

\allowdisplaybreaks

\def\epsilon{\varepsilon}
\def\phi {\varphi}

\providecommand{\abs}[1]{\left\lvert#1\right\rvert}
\providecommand{\norm}[1]{\left\lVert#1\right\rVert}

\renewcommand{\leq}{\leqslant}
\renewcommand{\geq}{\geqslant}
\providecommand{\abs}[1]{\left\lvert#1\right\rvert}
\providecommand{\norm}[1]{\left\lVert#1\right\rVert}

\title{Global uniqueness in an inverse problem for time fractional diffusion equations}
\author{Y. Kian, L. Oksanen, E. Soccorsi, M. Yamamoto}

\date{}
\begin{document}
\maketitle

\baselineskip 18pt
\begin{abstract}
Given $(M,g)$, a compact connected Riemannian manifold of dimension $d \geq 2$, with boundary $\partial M$, we consider an initial boundary value problem for a fractional diffusion equation on $(0,T) \times M$, $T>0$, with time-fractional Caputo derivative of order $\alpha \in (0,1) \cup (1,2)$. We prove uniqueness in the inverse problem of determining the smooth manifold $(M,g)$ (up to an isometry), and various time-independent smooth coefficients appearing in this equation, from measurements of the solution on a subset of $\pd M$ at fixed time. In the ``flat" case where $M$ is a compact subset of $\R^d$, two out the three coefficients $\rho$ (weight), $a$ (conductivity) and $q$ (potential) appearing in the equation $\rho \pd_t^\alpha u-\mdiv{a \nabla u}+ q u=0$ on $(0,T)\times \Omega$ are recovered simultaneously.


\end{abstract}
\section{Introduction}
\subsection{Statement of the problem}

Let $(M,g)$ be a compact connected Riemannian manifold of dimension $d \geq 2$, with boundary $\partial M$. 
For a strictly positive function $\mu$ we consider the weighted Laplace-Beltrami operator
$$
\Delta_{g,\mu} := \mu^{-1} \div_g\, \mu\, \nabla_g,
$$
where $\div_g$ (resp., $\nabla_g$) denotes the divergence (resp., gradient) operator on $(M,g)$,
and $\mu^{\pm 1}$ stands for the multiplier by the function $\mu^{\pm 1}$.
If $\mu$ is identically $1$ in $M$ then $\Delta_{g,\mu}$ coincides with the usual 
Laplace-Beltrami operator on $(M,g)$.
In local coordinates, we have
$$ \Delta_{g,\mu} u = \sum_{i,j=1}^d \mu^{-1} |g|^{-1/2}\partial_{x_i}(\mu |g|^{1/2}g^{ij}\partial_{x_j} u),
\quad u \in C^\infty(M), $$
where $g^{-1} :=(g^{ij})_{1 \leq i,j \leq d}$ and $|g| :=\mbox{det}\ g$. For $\alpha \in (0,2)$ we consider the initial boundary value problem (IBVP)
\bel{eqM1}
\left\{ \begin{array}{rcll}  \pd_t^\alpha u-\Delta_{g,\mu} u + q u & = & 0, & \mbox{in}\ (0,T)\times M,\\  
u & = & f, & \mbox{on}\ (0,T)\times \partial M,\\  
\pd_t^k u(0,\cdot) & = & 0, & \mbox{in}\ M,\ k=0,...,m,
\end{array} \right.
\ee
with non-homogeneous Dirichlet data $f$. Here $m :=[\alpha]$ denotes the integer part of $\alpha$ and $\pd_t^\alpha$ is the Caputo fractional derivative of order $\alpha$ with respect to $t$, defined by
\bel{cap} 
\pd_t^\alpha u(t,x):=\frac{1}{\Gamma(m+1-\alpha)}\int_0^t(t-s)^{m-\alpha}\pd_s^{m+1} u(s,x) \dd s,\ (t,x) \in Q,
\ee
where $\Gamma$ is the usual Gamma function expressed as $\Gamma(z):=\int_0^{+\infty} e^{-t} t^{z-1} \dd t$ for all $z \in \C$ such that $\Re{z}>0$. 
The system \eqref{eq1} models anomalous diffusion phenomena. In the sub-diffusive case $\alpha \in (0,1)$, the first line in \eqref{eq1} is usually named fractional diffusion equation, while in the super-diffusive case $\alpha\in (1,2)$, it is referred as fractional wave equation. 

Given two non empty open subsets $\Gi$ and $\Go$ of $\pd M$, $T_0 \in (0,T)$ and $\alpha \in (0,2)$, we introduce the function space
$$ \mathcal H_{\mathrm{in},\alpha, T_0} := \{ f \in C^{[\alpha]+1}([0,T], H^{\frac{3}{2}}(\pd M));\ \supp{f} \subset (0,T_0) \times \Gi \}, $$
where we recall that $[\alpha]$ stands for the integer part of $\alpha$. 
As established in Section 2, problem \eqref{eqM1} associated with $f \in \mathcal H_{\mathrm{in},\alpha, T_0}$ is well posed and the partial Dirichlet-to-Neumann (DN) map
\bel{M-def-DN}
\Lambda_{M,g,\mu,q}:\ \mathcal H_{\mathrm{in},\alpha, T_0} \ni f \mapsto 
\partial_\nu u(T_0, \cdot)_{\vert \Go}
:= \sum_{i,j=1}^d g^{ij}\nu_i\partial_{x_j} u (T_0, \cdot)_{\vert \Go},
\ee
where $u$ denotes the solution to \eqref{eqM1} and $\nu$ is the outward unit normal vector field along the boundary $\pd M$, is linear bounded from $\mathcal H_{\mathrm{in},\alpha, T_0}$ into $L^2(\Go)$.

In this paper we examine the problem whether knowledge of $\Lambda_{M,g,\mu,q}$ determines the Riemannian manifold $(M,g)$, and the functions $\mu$ and $q$, uniquely.

\subsection{Physical motivations}
Recall that fractional diffusion equations  with time fractional derivatives 
of the form 
\eqref{eqM1} describe  several physical phenomena  related to anomalous 
diffusion  such as diffusion of substances in heterogeneous media, 
diffusion of 
fluid flow in inhomogeneous anisotropic porous media, turbulent plasma, diffusion of carriers in amorphous photoconductors, diffusion in a
turbulent flow, a percolation
model in porous media, fractal media, various biological phenomena and finance 
problems (see  \cite{CSLG}). 
In particular, it is known (e.g., \cite{AG}) that
the classical diffusion-advection equation does not often 
interpret field data of diffusion of substances in the soil, and 
as one model equation, the fractional diffusion equation is used.

The diffusion equation with time fractional derivative is a corresponding 
macroscopic model equation to the continuous-time random walk (CTRW in short) 
and is derived from the CTRW (e.g., \cite{MK,RA}).

In particular, in the case where we consider fractional diffusion equations 
describing the diffusion of contaminants in a soil, we cannot a priori 
know governing parameters in \eqref{eqM1} such as 
reaction rate of pollutants.  Thus for prediction of contamination,
we need to discuss our inverse problem of determining these 
parameters from measurements of the flux on $\Go$ at 
a fixed time $t=T_0$  associated to  Dirichlet inputs at $\Gi$.

\subsection{State of the art}

Fractional derivative, ordinary and partial, differential equations have attracted attention over the two last decades. See \cite{KSTruj,MR,SM, P} 
regarding fractional calculus, and \cite{AW, GM,L}, and references therein, for studies of partial differential equations 
with time fractional derivatives. More specifically, the well-posedness of problem \eqref{eqM1} with time-independent coefficients is examined in 
\cite{BY, GLYama, SY}, and recently, 
weak solutions to \eqref{eqM1} have been defined in \cite{KY}.

There is a wide mathematical literature for inverse
coefficients problems associated with the system \eqref{eqM1} when $\alpha =1$ or $2$. Without being exhaustive, we refer to \cite{BK,CK1,Cho,CK,CY,GK,Katchalov2004} for the parabolic case $\alpha=1$ and to \cite{Belishev1987,Belishev1992,BCY,BF,BJY,A,CLS,Ki1,Lassas2010,LO,Lassas2014,LL,Mo,SU1,SU2,SU3} for the hyperbolic case $\alpha=2$. In contrast to parabolic or hyperbolic inverse coefficient problems, there is only a few mathematical papers dealing with inverse problems associated with \eqref{eqM1} when $\alpha\in(0,1) \cup (1,2)$. In the one-dimensional case, \cite{CNYY} uniquely determines the fractional order $\alpha$ and a time-independent coefficient, by Dirichlet boundary measurements. For $d\geq2$, the fractional order $\alpha$ is recovered in \cite{HNWY} from pointwise measurements of the solution over the entire time span. In \cite{SY}, the authors prove stable determination of the time-dependent prefactor of the source term. In the particular case where $d=1$ and $\alpha=1 / 2$, using a specifically designed Carleman estimate for \eqref{eqM1}, \cite{CXY,YZ}  derive a stability estimate of a zero order time-independent coefficient, with respect to partial internal observation of the solution. In \cite{LIY}, time-independent coefficients are uniquely identified by the Dirichlet-to-Neumann map obtained by probing the system with inhomogeneous Dirichlet boundary conditions of the form $\lambda(t)g(x)$, where $\lambda$ is a fixed real-analytic positive function of the time variable. Recently, \cite{FK} proved unique determination of a time-dependent parameter appearing in the source term or in a zero order coefficient, from pointwise measurements of the solution over the whole time interval.

\subsection{Main results}
The paper contains two main results. Both of them are uniqueness results for inverse coefficients problems associated with \eqref{eqM1}, but related to two different settings. In the first one, $(M,g)$ is a known compact subset of $\R^d$, while in the second one, $(M,g)$ is an unknown Riemannian manifold to be determined. 
The first setting is not contained in the second one, however, since 
in the second case, $(M,g)$ and all the other unknown coefficients are assumed to be smooth, while in the first case the regularity assumptions are relaxed considerably.

We begin by considering the case of a connected bounded domain $\Omega$ in $\R^d$, $d \geq 2$, with $C^{1,1}$ boundary $\pd \Omega$. Let $\rho \in C(\overline{\Omega})$, $V \in L^{\infty}(\Omega)$ and $a \in C^1(\overline{\Omega})$ fulfill the condition
\bel{t1a}
\rho(x) \geq c,\ a(x) \geq c,\ V(x) \geq 0,\ x \in \Omega,
\ee
for some positive constant $c$. For $M := \overline \Omega$, put
\bel{g_mu_from_a_rho}
g := \rho a^{-1}I_d,\ \mu := \rho^{1-d/2} |a|^{1/2},\ \mbox{and}\ q :=\rho^{-1}V,
\ee
in the first line of \eqref{eqM1}, where $I_d$ denotes the identity matrix in $\R^{d^2}$. Since
$(M,g)$ is a Riemannian manifold with boundary such that
$\mu |g|^{1/2} = \rho$, $g^{ij}=0$ if $i \neq j$, and $g^{ii}= \rho^{-1} a$ for $i , j \in \{1,\ldots,d \}$,
, we have
$$
\Delta_{g,\mu} u = \rho^{-1} \div (a \nabla u),\ u \in C^\infty(\overline \Omega).
$$
Therefore, \eqref{eqM1} can be equivalently rewritten as
\bel{eq1}
\left\{ \begin{array}{rcll} 
\rho \pd_t^\alpha u-\mdiv{a \nabla u}+ V u & = & 0, & \mbox{in}\ Q := (0,T) \times \Omega,\\  
u & = & f, & \mbox{on}\ \Sigma := (0,T) \times \pd \Omega,\\  
\partial_t^ku(0,\cdot) & = & 0, & \mbox{in}\ \Omega,\ k=0,\ldots,m.
\end{array} \right.
\ee
As will appear in Section \ref{sec-direct} for any arbitrary $\alpha\in (0,1) \cup (1,2)$ and $T_0 \in (0,T)$, the partial DN map  
\bel{def-DN}
\Lambda_{\rho,a,V}:\ \mathcal H_{\mathrm{in},\alpha, T_0} \ni f \mapsto a \partial_\nu u(T_0, \cdot)_{\vert \Go} := a \nabla u(T_0,\cdot) \cdot \nu_{\vert \Go}
\ee
where $u$ is the solution to \eqref{eq1} and $\nu$ is the outward unit normal vector to $\pd\Omega$, is bounded from $\mathcal H_{\mathrm{in},\alpha, T_0}$ into $L^2(\Go)$. Our first result claims that knowledge of $\Lambda_{\rho,a,V}$ uniquely determines
two out of the three coefficients $\rho$, $a$, and $V$, which are referred as, respectively, the density, the conductivity, and the (electric) potential.

\begin{thm}
\label{thm-main}
Assume that $\Gi \cap \Go \neq \emptyset$ and that $\Gi \cup \Go=\pd\Omega$. For $j=1,2$, let $\rho_j \in L^{\infty}(\Omega)$, $a_j \in  W^{2,\infty}(\Omega)$, and $V_j \in L^{\infty}(\Omega)$ satisfy \eqref{t1a} with $\rho=\rho_j$, $a=a_j$, $V=V_j$. Moreover, let either of the three following conditions be fulfilled:
\begin{enumerate}[(i)]
\item $\rho_1=\rho_2$ and
\bel{t1c}
\nabla a_1(x)=\nabla a_2(x),\ x \in \pd \Omega.
\ee
\item $a_1=a_2$ and
\bel{t1b}
\exists C>0,\ | \rho_1(x)-\rho_2(x) | \leq C \mathrm{dist}(x,\pd \Omega)^2,\ x \in \Omega.
\ee
\item $V_1=V_2$ and \eqref{t1c}-\eqref{t1b} hold simultaneously true.
\end{enumerate}
Then, $\Lambda_{\rho_1,a_1,V_1}=\Lambda_{\rho_2,a_2,V_2}$ yields $(\rho_1,a_1,V_1)=(\rho_2,a_2,V_2)$.
\end{thm}

The second result describes the identifiability properties of the Riemannian manifold $(M,g)$ and the functions $\mu \in \mathcal C^\infty(M)$ and $q \in \mathcal C^\infty(M)$, appearing in the first line of the IBVP \eqref{eqM1}, that can be inferred from $\Lambda_{M,g,\mu,q}$.
It is well known that the DN map is invariant under isometries fixing the boundary. Moreover, gauge equivalent coefficients $(\mu, q)$ cannot be distinguished by the DN map either. Here and henceforth, $(\mu_1,q_1)$ and $(\mu_2,q_2)$ are said gauge equivalent if there exists a strictly positive valued function $\kappa \in C^\infty(M)$ satisfying
\begin{align}
\label{kappa_near_bd}
\kappa(x) = 1
\  \mbox{and} \ 
\pd_\nu \kappa(x) =0,\ x \in \pd M
\end{align}
such that
\begin{align}
\label{gauge_mu_q}
 \mu_2 = \kappa^{-2} \mu_1,
\quad  q_2 = q_1 - \kappa \Delta_{g, \mu_1} \kappa^{-1}.
\end{align}

Our second statement in as follows.

\begin{thm}
\label{t5} 
For $j=1,2$, let $(M_j,g_j)$ be two compact and smooth connected Riemannian manifolds of dimension $d \geq 2$ with the same boundary, and let $\mu_j \in C^\infty(M_j)$ and $q_j \in C^\infty(M_k)$
satisfy $\mu_j(x) > 0$ and $q_j(x) \ge 0$
for all $x \in M_j$.
Let $\Gi, \Go \subset \pd M_1$
be relatively open
and suppose that $\overline{\Gi} \cap \overline{\Go} \ne \emptyset$.
Suppose, moreover, that $g_1 = g_2$, $\mu_1 = \mu_2=1$
and $\pd_\nu \mu_1 = \pd_\nu \mu_2=0$ on $\pd M_1$.
Then, $\Lambda_{M_1,g_1,\mu_1,q_1}= \Lambda_{M_2,g_2,\mu_2,q_2}$ yields that
$(M_1, g_1)$ and $(M_2, g_2)$ are isometric
and that $(\mu_1, q_1)$ and $(\mu_2, q_2)$ are gauge equivalent.
\end{thm}

\subsection{Comments}

Notice that the absence of global uniqueness result manifested in Theorems \ref{thm-main} (in the sense that only two of the three coefficients $\rho$, $a$, and $V$, are recovered) and \ref{t5} (where the metric $g$ is determined up to an isometry and $(\mu,q)$ are identified modulo gauge transformation) arises from one or several natural obstructions to identification in the system under investigation, each of them being induced by an invariance property satisfied by \eqref{eqM1}.

The first obstruction, which can be found both in Theorems \ref{thm-main} and \ref{t5}, is due to the invariance of \eqref{eqM1} under the group of gauge transformations given by 
\eqref{gauge_mu_q}. Indeed, given a strictly positive function $\kappa \in C^\infty(M)$ satisfying \eqref{kappa_near_bd}, we observe for any $(\mu_1,q_1)$ and $(\mu_2,q_2)$ obeying \eqref{gauge_mu_q}, that 
\begin{align*}
\Delta_{g, \mu_2} (\kappa w)
= 
\kappa \Delta_{g, \mu_1} w  + \delta \kappa w,\ w \in C^\infty(M),
\end{align*}
where $\delta := \kappa^{-1} \Delta_{g, \mu_1} \kappa
- 2 \kappa^{-2} (\nabla_g \kappa, \nabla_g \kappa)_g$, and $(\cdot, \cdot)_g$ denotes
the inner product on $(M,g)$.
In particular, taking $w = \kappa^{-1}$ we get 
the simpler expression
$\delta = - \kappa \Delta_{g, \mu_1} \kappa^{-1}$.
Finally, taking $w = u$, where $u$ is the solution to \eqref{eqM1} associated with $\mu=\mu_1$ and $q=q_1$, we find that
$$
(\pd_t^\alpha - \Delta_{g, \mu_2} +  q_2) (\kappa u) 
= \kappa(\pd_t^\alpha - \Delta_{g,\mu_1} + q_1)u = 0.
$$
Since our assumptions \eqref{kappa_near_bd} on $\kappa$ imply that $\pd_\nu (\kappa u) = \pd_\nu u$ and $\kappa u=u$ on $(0,T) \times \pd M$, we find that $\Lambda_{M,g,\mu_1,q_1} = \Lambda_{M,g,  \mu_2,  q_2}$.
This proves that the DN map is invariant under the group of gauge transformations $$(\mu, q) \mapsto ( \kappa^{-2}\mu, q - \kappa \Delta_{g, \mu} \kappa^{-1})$$ 
parametrized by strictly positive functions
$\kappa \in C^\infty(M)$ satisfying \eqref{kappa_near_bd}. 
Notice that the conditions $g_1 = g_2$, $\mu_1 = \mu_2=1$ and $\pd_\nu \mu_1 = \pd_\nu \mu_2=0$ imposed on $\pd M_1$ in 
Theorem \ref{t5} are analogous to \eqref{t1b} in Theorem \ref{thm-main}. Moreover, the above mentioned invariance property of the system indicates that
the result of Theorem \ref{thm-main}, where two of the three coefficients $\rho$, $a$, and $V$, are simultaneously identified while keeping the third one fixed, is the best one could expect.

The second obstruction arises from the fact that \eqref{eqM1}
is invariant with respect to changes of coordinates.
That is, if $\Phi : M \to M$ is a diffeomorphism fixing the boundary $\pd M$ then 
$\Lambda_{M,g,\mu,q} = \Lambda_{M,\Phi^* g,\mu,q}$
where $\Phi^* g$ is the pullback of $g$ by $\Phi$.

To our best knowledge, the results of this article are the most precise so far, about the recovery of coefficients appearing in a time fractional diffusion equation from boundary measurements. We prove recovery of a wide class of coefficients from partial boundary measurements that consist of an input on the part $\Gi$ of the boundary and observation of the flux at the part $\Go$ for one fixed time $t=T_0 \in (0,T)$.
Our results extend the ones contained in the previous works \cite{CNYY,CXY,HNWY,LIY,YZ} related to this problem. Another benefit of our approach is its generality, which makes it possible to treat the case of a smooth Riemannian manifold, and the one of a bounded domain with weak regularity assumptions on the coefficients.

Notice that \eqref{eq1} associated with $\alpha=1$ is the usual heat equation, in which case Theorem \ref{thm-main} is contained \cite{CK1,CK2}.
We point out that the strategy used in \cite{CK1,CK2} for the derivation of Theorem \ref{thm-main} with $\alpha=1$, cannot be adapted to the framework of time fractional derivative diffusion equations of order $\alpha \in (0,1) \cup (1,2)$. This is due to the facts that a solution to a time fractional derivative equation is not described by a semi-group, and that there is only limited smoothing property, and no integration by parts formula or Leibniz rule, with respect to the time variable, in this context. As a consequence, the analysis developped in this text is quite different from the one carried out by \cite{CK1,CK2}.

Notice from Theorem \ref{t5} that the statement of Theorem \ref{thm-main} still holds true for smooth coefficients in a smooth domain, under the weaker assumption $\overline{\Gi} \cap \overline{\Go} \ne \emptyset$. Nevertheless, in contrast to Theorem \ref{t5} where we focus on the recovery of the Riemaniann manifold and the metric, the main interest of Theorem \ref{thm-main} lies in the weak regularity assumptions imposed on the unknown coefficients of the inverse problem under consideration. In the same spirit, we point out with Theorem \ref{t6} below, that the result of Theorem \ref{t5} remains valid when $\overline{\Gi} \cap \overline{\Go} = \emptyset$, in the special case where $\mu=1$ and $q=0$, and assuming a Hassell-Tao type inequality \cite{Hassell2002}. 

The key idea to our proof is the connection between the DN map associated with \eqref{eqM1} and the boundary spectral data of the corresponding elliptic Schr\"odinger operator. This ingredient has already been used by several authors in the context of hyperbolic (see e.g. \cite{Katchalov2001, Katchalov2004, Lassas2010, LO, Lassas2014}), parabolic (see e.g. \cite{CK2, Katchalov2004}), and dynamical Schr\"odinger (see e.g. \cite{Katchalov2004}) equations. Nevertheless, to our best knowledge, there is no such approach for time fractional diffusion equations, available in the mathematical literature. Once the connection between the DN map and the boundary spectral data is established, we obtain Theorems \ref{thm-main} and \ref{t5} by applying a Borg-Levinson type inverse spectral result (see e.g. \cite{CK2,CS,Katchalov1998,Katchalov2001,KKS,Ki,NSU}).

\subsection{Outline} 
The paper is organized as follows. The next three sections are devoted to the study of the inverse problem associated with \eqref{eq1} in a bounded domain,
while the last section contains the analysis of the inverse problem associated with \eqref{eqM1} on a Riemannian manifold. 

More precisely, we establish a connection between Theorem \ref{thm-main} and a Borg-Levinson type inverse spectral result in the first part of Section \ref{sec-set}. In the second part, we introduce mathematical tools used in the analysis of the direct problem associated with \eqref{eq1}, which is carried out in Section \ref{sec-direct}. Then we define the partial DN map $\Lambda_{\rho,a,V}$ at the end of Section \ref{sec-direct}, and complete the proof of Theorem  \ref{thm-main} in Section \ref{sec-proof}. Finally, Section \ref{sec-Riemann}
contains the proofs of Theorem \ref{t5}, and the stronger result stated in Theorem \ref{t6} in the particular case where $\mu=1$ and $q=0$.


\section{The settings}
\label{sec-set}
In this section, we begin the analysis of the inverse problem associated with \eqref{eq1}, which is the purpose of Theorem \ref{thm-main}.
We first establish the connection between Theorem \ref{thm-main} and a suitable version of the Borg-Levinson inverse spectral theorem. 

\subsection{Borg-Levinson type inverse spectral problem and Theorem \ref{thm-main}}

Let us start by defining the boundary spectral data used by the Borg-Levinson type inverse spectral theory.

\noindent \textbf{Boundary spectral data.}
Given a positive constant $c$, we assume that $\rho \in L^{\infty}(\Omega)$ satisfies $\rho(x) \geq c >0$ for a.e. $x \in \Omega$, so the scalar product 
$$ \langle u , v \rangle_\rho :=\int_\Omega \rho(x) u(x) v(x) \dd x,\ u , v \in L^2(\Omega), $$ 
is equivalent to the usual one in $L^2(\Omega)$. We denote by $L_\rho^2(\Omega)$ the Hilbertian space $L^2(\Omega)$ endowed with $\langle \cdot , \cdot \rangle_\rho$.

Next, for a nonnegative $V \in L^{\infty}(\Omega)$, and for $a \in C^1(\overline{\Omega})$ fulfilling $a(x) \geq c > 0$ for every $x \in \Omega$, we introduce the quadratic form
$$
h[u] := \int_{\Omega} \left( a(x) | \nabla u(x) |^2 + V(x) | u(x) |^2 \right) \dd x,\ u \in \dom(h) := H_0^1(\Omega), 
$$
and consider the operator $H$ generated by $h$ in $L_\rho^2(\Omega)$. Since $\pd \Omega$ is $C^{1,1}$, $H$ is self-adjoint in $L_\rho^2(\Omega)$ and acts on its domain as
\bel{op-h} 
H u := \rho^{-1} \left( \mdiv{a \nabla u}  + V u \right),\ u \in \dom (H) := H_0^1(\Omega) \cap H^2(\Omega),
\ee
according to \cite[Theorem 2.2.2.3]{G}.

By the compactness of the embedding $H_0^1(\Omega) \hookrightarrow L_\rho^2(\Omega)$, the spectrum $\sigma(H)$ of the operator $H$ is purely discrete. Let $\{ \lambda_n;\, n \in \N^* \}$ be the non-decreasing sequence of the eigenvalues (repeated according to multiplicities) of $H$. Furthermore, we introduce a family $\{ \varphi_n;\ n \in \N^* \}$ of eigenfuctions of the operator $H$, which satisfy 
\bel{eq1b}
H \varphi_n = \lambda_n \varphi_n,\ n \in \N^*,
\ee
and form an orthonormal basis in $L_\rho^2(\Omega)$. Notice that each $\varphi_n$ is a solution to the following Dirichlet problem :
\bel{eq2}
\left\{ \begin{array}{rcll}
-\mdiv{a \nabla \varphi_n} + V \varphi_n &= & \lambda_n \rho \varphi_n, & \mbox{in}\ \Omega,\\ 
\varphi_n & = & 0, & \mbox{on}\ \partial\Omega,\\ 
\int_{\Omega} \rho(x) \abs{\varphi_n(x)}^2 \dd x &=& 1,\ & 
\end{array} 
\right.
\ee
Put $\psi_n := \left( a \partial_\nu {\varphi_n} \right)_{|\pd \Omega}$ for every $n \in \N^*$. Following \cite{CK1, CK2,Katchalov2001}, we define the boundary spectral data associated with $(\rho,a,V)$, as
$$ \mathrm{BSD}(\rho,a,V) :=\{ (\lambda_n, \psi_n );\ n\geq1\}. $$
\textbf{A strategy for the proof of Theorem \ref{thm-main}.}
We first recall from \cite[Corollaries 1.5, 1.6 and 1.7]{CK2} the following Borg-Levinson type theorem.

\begin{prop}
\label{p1}
Under the conditions of Theorem \ref{thm-main}, assume that either of the three assumptions (i), (ii) or (iii) is verified.
Then $\mathrm{BSD}(\rho_1,a_1,q_1)=\mathrm{BSD}(\rho_2,a_2,q_2)$ entails that $(\rho_1,a_1,q_1)=(\rho_2,a_2, q_2)$.
\end{prop}

In view of the inverse spectral result stated in Proposition \ref{p1}, we may derive the claim of Theorem \ref{thm-main} upon showing that two sets of admissible coefficients $(\rho_j,a_j,V_j)$, $j=1,2$, have same boundary spectral data, provided their boundary operators $\Lambda_{\rho_j,a_j,V_j}$ coincide. Otherwise stated, the proof of Theorem \ref{thm-main} is a byproduct of Proposition \ref{p1} combined with the coming result :

\begin{thm}
\label{t2} 
For $j=1,2$, let $V_j \in L^\infty(\Omega)$, $\rho_j \in L^\infty(\Omega)$ and $a_j \in \mathcal C^{1}(\overline{\Omega})$ satisfy \eqref{t1a} with $\rho=\rho_j$, $a=a_j$, $V=V_j$. Then $\Lambda_{\rho_1,a_1,V_1}= \Lambda_{\rho_2,a_2,V_2}$ implies
$\mbox{BSD}(\rho_1,a_1,V_1)=\mbox{BSD}(\rho_2,a_2,V_2)$, up to an appropriate choice of the eigenfunctions of the operator $H_1$ defined in \eqref{op-h} and associated with $(\rho,a,V)=(\rho_1,a_1,V_1)$.
\end{thm}

Therefore, we are left with the task of proving Theorem \ref{t2}. 

\subsection{Technical tools}
\ \\
\textbf{Fractional powers of $H$.}
Since $H$ is a strictly positive operator, for all $s\geq 0$, we can define $H^s$  by 
\[H^s h=\sum_{n=1}^{+\infty}\left\langle h,\phi_n\right\rangle \lambda_n^s\phi_n,\quad h\in D(H^s)=\left\{h\in L^2(\Omega):\ \sum_{n=1}^{+\infty}\abs{\left\langle h,\phi_n\right\rangle}^2 \lambda_n^{2s}<\infty\right\}\]
and we consider on $D(H^s)$ the norm
\[\|h\|_{D(H^s)}=\left(\sum_{n=1}^{+\infty}\abs{\left\langle h,\phi_n\right\rangle}^2 \lambda_n^{2s}\right)^{\frac{1}{2}},\quad h\in D(H^s).\]
\textbf{Two parameters Mittag-Leffler function.}
Let $\alpha$ and $\beta$ be two positive real numbers. Following \cite[Section 1.2.1, Eq. (1.56)]{P}, we define the Mittag-Leffler function associated with $\alpha$ and $\beta$, by the series expansion
\bel{def-MLf}
E_{\alpha,\beta}(z)=\sum_{n=0}^{+\infty} \frac{z^n}{\Gamma(\alpha n + \beta)},\ z \in \C.
\ee
In the particular framework of this paper, where $\alpha \in (0,2)$, we recall for further reference from \cite[Theorem 1.4]{P} the three following useful estimates.

The first estimate, which holds for any $\alpha \in (0,2)$ and $\beta \in \R_+^*$, claims that there exists a constant $c>0$, depending only on $\alpha$ and $\beta$, such that we have
\bel{es0}
| E_{\alpha,\beta}(-t) | \leq \frac{c}{1+ t},\ t \in \R_+^*.
\ee

The second estimate applies for $\alpha=\beta \in (0,2)$ and states for every $\theta \in (\pi \alpha \slash 2, \pi \alpha)$, that
\bel{es1}
| E_{\alpha,\alpha}(z) | \leq \frac{c}{1+ |z|^{2}},
\ee
whenever $z \in \C \setminus \{ 0 \}$ satisfies $\ | \mbox{arg}(z) | \in [\theta, \pi]$. Here $c$ is a positive constant depending only on $\alpha$ and $\theta$.  In contrast to  \eqref{es0}, which is explicitly stated at formula (1.148) of \cite[Theorem 1.6]{P}, estimate \eqref{es1} follows from the asymptotic behavior of $E_{\alpha,\alpha}(z)$ as $|z|\to+\infty$ given by formula (1.143) of \cite[Theorem 1.4]{P}. Indeed, formula (1.143) of \cite[Theorem 1.4]{P} implies that for $\ | \mbox{arg}(z) | \in [\theta, \pi]$ we have
\[E_{\alpha,\beta}(z)= -{z^{-1}\over \Gamma(\alpha-\beta)}+\underset{|z|\to+\infty}{\mathcal O}\left(|z|^{-2}\right)\]
and using the fact that ${z^{-1}\over \Gamma(\alpha-\beta)}=0$ for $\alpha=\beta$ we deduce \eqref{es1}.

Finally, the third estimate we shall need in the derivation of Theorem \ref{thm-main}, follows readily from \eqref{es1} and reads:
\bel{es2}
| E_{\alpha,\alpha}(-t) | \leq \frac{c}{1+ t^{2 }},\ t \in \R_+^*.
\ee
\textbf{An {\it a priori} elliptic estimate.} In what follows, we shall make use several times of the following result.

\begin{lem}
\label{lm-elliptic}
Let $\rho \in L^{\infty}(\Omega)$, $a \in C^1(\overline{\Omega})$ and $V \in L^{\infty}(\Omega)$ fulfill \eqref{t1a}. Then there exists a constant $c>0$, such that the estimate
\bel{ell1c}
\sum_{n=1}^{+\infty} \lambda_n^{-2} \left| \int_{\pd \Omega} g(x) \psi_n(x) \dd x \right|^2 \leq c \| g \|_{H^{3 / 2}(\pd \Omega)}^2, 
\ee
holds whenever $g \in H^{3 / 2}(\pd \Omega)$.
\end{lem} 
\begin{proof}
We first prove that the boundary value problem
\bel{ell1}
\left\{ \begin{array}{rcll} -\mdiv{a \nabla v}+ V v & = & 0, & \mbox{in}\ \Omega,\\  
v & =& g, & \mbox{in}\ \pd \Omega,
\end{array} \right.
\ee
admits a unique solution $v \in H^2(\Omega)$. To do that, we refer to \cite[Section 1, Theorem 8.3]{LM1} and pick $G \in H^2(\Omega)$ satisfying $G=g$ on $\pd \Omega$ and the estimate
\bel{ell2}
\| G \|_{H^2(\Omega)} \leq C \| g \|_{H^{3 \slash 2}(\pd \Omega)}.
\ee
Here and in the remaining of the proof $C$ denotes a positive constant that does not depend on $g$. Evidently $v$ is a solution to \eqref{ell1} if and only if $w=v-G$ is a solution to
\bel{ell3}
\left\{ \begin{array}{rcll} -\mdiv{a \nabla w}+ V w & = & F_G, & \mbox{in}\ \Omega,\\  
w & =& 0, & \mbox{in}\ \pd \Omega,
\end{array} \right.
\ee
where $F_G:=-\left( -\mdiv{a \nabla G}+ V G \right)$. Since $H$ is boundedly invertible in $L^2(\Omega)$ and $\rho^{-1} F_G \in L^2(\Omega)$, then 
$w = H^{-1} ( \rho^{-1} F_G ) \in \dom(H)=H_0^1(\Omega) \cap H^2(\Omega)$ is the unique solution to \eqref{ell3}, and we have
\bel{ell4}
\| w \|_{H^2(\Omega)} \leq C \| F_G \|_{L^2(\Omega)}.
\ee
Here we used the fact, arising from the strict ellipticity of $H$ 
(see \cite[Sections 2.2, 2.3, and 2.4]{G}), that the graph norm of $H$ is equivalent to the usual norm in $H^2(\Omega)$. Therefore, $v=w+G \in H^2(\Omega)$ is the unique solution to\eqref{ell1} and we derive from \eqref{ell2} and \eqref{ell4} that
\bel{ell5}
\| v \|_{H^2(\Omega)} \leq C \| g \|_{H^{3 / 2}(\pd \Omega)}.
\ee

Now, in view of \eqref{ell1}, we get for each $n \in \N^*$ that
$$ 
0 = \langle -\mdiv{a \nabla v}+ V v  , \varphi_n \rangle = \langle v, H \varphi_n \rangle_\rho + \int_{\pd \Omega} g(x) \psi_n(x) \dd x,
$$
upon integrating by parts twice. This and \eqref{eq1b} yield that
\bel{ell1b}
v_n:=\langle v , \varphi_n \rangle_\rho = - \lambda_n^{-1} \int_{\pd \Omega} g(x) \psi_n(x) \dd x,\ n \in \N^*. 
\ee
Finally, putting the Parseval identity $\sum_{n=1}^{+\infty} | v_n |^2=\| v \|_{\rho}^2$ together with \eqref{ell5}-\eqref{ell1b}, we obtain \eqref{ell1c}.
\end{proof}

\section{Analysis of the direct problem}
\label{sec-direct}
In this section we rigorously define the DN map \eqref{def-DN}, which requires that the direct problem associated with \eqref{eq1} be preliminarily examined. Next we relate the DN map \eqref{def-DN} to the BSD.

We start with the sub-diffusive case $\alpha \in (0,1)$.
\begin{prop}
\label{pr-solution}
Let $\alpha\in(0,1)$, $\rho \in L^{\infty}(\Omega)$, $a \in C^1(\overline{\Omega})$ and $V \in L^{\infty}(\Omega)$ fulfill \eqref{t1a}, and let
$f \in C^1([0,T]; H^{\frac{3}{2}}(\pd \Omega))$ satisfy $f(0,\cdot)=0$ in $\partial\Omega$. 
Then, there exists a unique solution $u \in C([0,T],L^{2}(\Omega))$ to the boundary value problem \eqref{eq1}. Moreover, we have $u \in C((0,T], H^{2\gamma}(\Omega))$ for any $\gamma \in (0,1)$.
\end{prop}
\begin{proof}
With reference to \cite[Section 1, Theorem 8.3]{LM1} we pick $F \in C^1([0,T],H^2(\Omega))$ satisfying $F=f$ on $\Sigma$.
Then, it is apparent that $u$ is a solution to \eqref{eq1} if and only if $v:=u-F$ is a solution to the IBVP
\bel{eu2}
\left\{ \begin{array}{rcll} 
\rho \pd_t^\alpha v-\mdiv{a \nabla v}+ V v & = & G, & \mbox{in}\ Q,\\  
v & =& 0, & \mbox{on}\ \Sigma,\\  
v(0,\cdot) & = & v_0, & \mbox{in}\ \Omega,
\end{array} \right.
\ee
where $G:=-( \rho \pd_t^{\alpha} F - \nabla \cdot a \nabla F + V F)$ and $v_0:=-F(0,\cdot)$. Applying the Laplace transform to \eqref{eu2} we find through basic computations similar to the ones used in the derivation of \cite[Theorems 2.1 and 2.2]{SY}, that 
\bel{eu3}
v(t,\cdot) = S_0(t) v_0 + \int_0^t S(s) G(t-s,\cdot) \dd s,\ t \in (0,T), 
\ee
where we have set
\bel{eu3c}
S_0(t) h := \sum_{n=1}^{+\infty} E_{\alpha,1}(- \lambda_n t^{\alpha}) \langle h , \varphi_n \rangle \varphi_n\ \mbox{and}\
S(t) h := t^{\alpha-1} \sum_{n=1}^{+\infty} E_{\alpha,\alpha}(- \lambda_n t^{\alpha}) \langle h , \varphi_n \rangle \varphi_n,
\ee
for every $t \in (0,T)$ and $h \in L^2(\Omega)$. Further, in view of \eqref{es0}, we have
\bea
 \| S_0(t) h \|_{H^{2}(\Omega)}^2 &  = &  \sum_{n=1}^{+\infty} \lambda_n^{2} E_{\alpha,1}(-\lambda_n t^{\alpha})^2 \left| \langle h , \varphi_n \rangle \right|^2 \nonumber\\
& \leq & C \sum_{n=1}^{+\infty} t^{-2 \alpha} \left| \langle h , \varphi_n \rangle \right|^2 = C t^{-2 \alpha} \| h \|_{L^2(\Omega)}^2,  \label{eu3d}
\eea
where $C$ is a positive constant that is independent of $t$ and $h$. The convergence of the series appearing in the right hand side of \eqref{eu3d} being uniform with respect to $t \in [\varepsilon,T]$, for any fixed $\varepsilon \in (0,T)$, then $t \mapsto S_0(t) h \in C([\varepsilon,T],H^{2}(\Omega))$. And since $\varepsilon$ is arbitrary in $(0,T)$, we ned up getting that
\bel{eu3e}
t \mapsto S_0(t) h \in C((0,T],H^{2}(\Omega)).
\ee
Similarly, we obtain for all $t \in (0,T)$ and $h \in L^2(\Omega)$ that
\bea
\| S(t) h \|_{H^{2 \gamma}(\Omega)}^2 & = & t^{2(\alpha-1)} \sum_{n=1}^{+\infty} \lambda_n^{2 \gamma} E_{\alpha,\alpha}(-\lambda_n t^{\alpha})^2 \left| \langle h , \varphi_n \rangle \right|^2 \nonumber \\
& \leq & C \sum_{n=1}^{+\infty} t^{-2(1-\alpha (1-\gamma))} \left| \langle h , \varphi_n \rangle \right|^2 = C t^{-2(1-\alpha (1-\gamma))} \| h \|_{L^2(\Omega)}^2. \label{eu4}
\eea
As a consequence we have $t \mapsto \int_0^t S(s) G(t-s,\cdot) \dd s \in C([0,T],H^{2 \gamma}(\Omega))$, since $G \in C([0,T],L^2(\Omega))$. This, \eqref{eu3} and \eqref{eu3e} yield that $v$, and hence $u$, is lying in $C((0,T],H^{2 \gamma}(\Omega))$. 

We turn now to proving that $\lim_{t \to 0} \| u(t) \|_{L^2(\Omega)}=0$, or equivalently that $\lim_{t \to 0} \| v(t) - v_0 \|_{L^2(\Omega)}=0$.
With reference to \eqref{eu3}, we shall actually establish that 
\bel{eu5} 
\lim_{t \to 0} \| S_0(t) v_0 - v_0 \|_{L^2(\Omega)} = \lim_{t \to 0} \left\| \int_0^t S(s) G(t-s) \dd s \right\|_{L^2(\Omega)}=0.
\ee
This can be achieved upon recalling from \eqref{eu3c} that 
\bel{eu6} 
\| S_0(t) v_0 - v_0 \|_{L^2(\Omega)}^2 = \sum_{n=1} \left( E_{\alpha,1}(-\lambda_{n} t^{\alpha}) -1 \right)^2 | \langle v_0 , \varphi_n \rangle |^2,\ t \in (0,T),
\ee
noticing that $\underset{t \to 0}{\lim} \left( E_{\alpha,1}(-\lambda_{n} t^{\alpha}) -1 \right) =0$ for every $n \in \N^*$,
and taking advantage of the fact that the series in the right hand side of \eqref{eu6} convergences uniformly with respect to $t \in (0,T)$, as we have
$$  | E_{\alpha,1}(-\lambda_{n} t^{\alpha}) -1 | \leq \frac{c}{1 + \lambda_{n} t^{\alpha}} + 1 \leq c + 1,\ t \in (0,T),\ n \in \N^*, $$
by \eqref{es0}. Further, since
\beas
\left\| \int_0^t S(s) G(t-s) \dd s \right\|_{L^2(\Omega)} & \leq & \int_0^t \| S(t) G(t-s) \|_{H^{2 \gamma}(\Omega)} \dd s \\
& \leq & C \int_0^t s^{-1+\alpha(1-\gamma)} \| G(t-s,\cdot) \|_{L^2(\Omega)} \dd s \\
& \leq & ( C \slash \alpha(1-\gamma) ) t^{\alpha(1-\gamma)} \| G \|_{C([0,T],L^2(\Omega))},\ t \in (0,T),
\eeas
by \eqref{eu4}, we end up getting \eqref{eu5}. This terminates the proof since $v$, and hence $u$, is uniquely defined by \eqref{eu3}.
\end{proof}
In view of Proposition \ref{pr-solution}, for $\alpha\in(0,1)$ and for any $f \in C^1([0,T], H^{\frac{3}{2}}(\pd \Omega))$ such that $f(0,\cdot)=0$ on $\pd \Omega$ and all $\gamma \in (0,1)$, there exists
a unique solution $u \in C([0,T], L^2(\Omega)) \cap C((0,T], H^{2\gamma}(\Omega))$ to \eqref{eq1}. Thus, taking $\gamma \in (3 \slash 4,1)$, we see that the mapping
$$ a \pd_\nu u:\ [0,T] \times \pd \Omega \ni (t,x)  \mapsto  a(x) \pd_\nu u(t,x):=a(x) \nabla u(t,x) \cdot \nu(x), $$
where $\nu$ denotes the outward unit normal vector to $\pd \Omega$, is well defined in $C((0,T],L^2(\partial\Omega))$. From this result we deduce for all $\alpha\in(0,1)$ that the operator $\Lambda_{\rho,a,V}$ is bounded from 
$\mathcal H_{\mathrm{in},\alpha,T_0}$ into $L^2(\Go)$.

Further, arguing as above, we derive the following result in the super-diffusive case $\alpha \in (1,2)$.
\begin{prop}
\label{pr-solution-wave}
Let $\alpha\in(1,2)$, $\rho \in L^{\infty}(\Omega)$, $a \in C^1(\overline{\Omega})$ and $q \in L^{\infty}(\Omega)$ fulfill \eqref{t1a}, and let
$f \in C^2([0,T]; H^{\frac{3}{2}}(\pd \Omega))$ satisfy $f(0,\cdot)=\partial_tf(0,\cdot)=0$ in $\partial\Omega$. 
Then, there exists a unique solution $u \in C([0,T],L^{2}(\Omega))$ to the boundary value problem \eqref{eq1}. Moreover, we have $u \in C((0,T], H^{2\gamma}(\Omega))$ for any $\gamma \in (0,1)$.\end{prop}

Fix $\alpha \in (1,2)$. We deduce from Proposition \ref{pr-solution-wave} that for all $\gamma \in (3/4,1)$ and all $f \in C^2([0,T], H^{\frac{3}{2}}(\pd \Omega))$ verifying $f(0,\cdot)=\pd_t f(0,\cdot)=0$ on $\pd \Omega$, there exists
a unique solution $u \in C([0,T], L^2(\Omega)) \cap C((0,T], H^{2\gamma}(\Omega))$ to \eqref{eq1}. Therefore, the mapping
$$ a \pd_\nu u:\ [0,T] \times \pd \Omega \ni (t,x)  \mapsto  a(x) \pd_\nu u(t,x), $$
 is well defined in $C((0,T],L^2(\partial\Omega))$, and the operator $\Lambda_{\rho,a,V}$ is bounded from $\mathcal H_{\mathrm{in},\alpha,T_0}$ into $L^2(\Go)$.


\subsection{Normal derivative representation formula}

In view of deriving the representation formula of $\Lambda_{\rho,a,V}$ given in Proposition \ref{pr-rf}, we start by establishing the following technical result..

\begin{lem}
\label{l1} Let $\alpha\in (0,1) \cup (1,2)$.
 For $\rho$, $a$ and $q$ as in Proposition \ref{pr-solution} and $f \in \mathcal H_{\mathrm{in},\alpha,T_0}$, the solution $u$ to \eqref{eq1} reads
\bel{l1a} 
u(t,\cdot)=\sum_{n=1}^\infty u_n(t) \varphi_n,
\ee
for each $t \in [0,T]$, where $u_n(t) := \langle u(t,\cdot) , \varphi_n \rangle_\rho$ expresses as
\bel{l1b}
u_n(t)=-\int_0^t  s^{\alpha-1} E_{\alpha,\alpha}(-\lambda_n s^\alpha) \left( \int_{\pd \Omega} f(t-s,x) \psi_n(x) \dd \sigma(x) \right) \dd s.
\ee
\end{lem}

\begin{proof}
The identity \eqref{l1a} follows readily from the fact that $u$ is lying in $C([0,T],L^2(\Omega))$ and that $\{ \varphi_n;\ n \in \N^* \}$ is an orthonormal basis of $L_\rho^2(\Omega)$.
Next, upon extending $f$ by zero outside $[0,T] \times \pd \Omega$, i.e. putting $f(t,x):=0$ for $(t,x) \in (T,+\infty) \times \pd \Omega$, and denoting by $u$ the solution to \eqref{eq1} in $(0,+\infty) \times \Omega$, we compute for all $n \in \N^*$, the Laplace transform $\cL[u_n](p)$ of $u_n$ which is well defined for  $p \in \R_+^*$ according to estimate \eqref{es1} and Proposition \ref{pr-solution}, \ref{pr-solution-wave}. We get
\bel{rf1}
\cL[u_n](p)=\int_0^{+\infty} u_n(t)e^{-pt} \dd t=\int_\Omega \rho(x) \left( \int_0^{+\infty} u(t,x) e^{-pt} \dd t \right) \phi_n(x) \dd x= 
\langle \cL[u](p,.) , \phi_n \rangle_\rho,
\ee
through standard computations.
Since $\cL[ \pd_t^{\alpha} u](p)=p^{\alpha} \cL[u](p)$, by \cite[Eq. (2.140)]{P} and the third line of \eqref{eq1}, we deduce from \eqref{rf1} upon applying the Laplace transform on both sides of the first line in \eqref{eq1}, that
\bel{rf2}
p^\alpha \cL[u_n](p)= \langle p^\alpha \cL [u](p,.), \phi_n \rangle_\rho = -\langle -\mdiv{a \nabla \cL[u](p)} + q \cL[u](p) , \phi_n \rangle,\ p \in \R_+^*.
\ee
Thus, applying the Green formula in the right hand side of \eqref{rf2}, we get for each $p \in \R_+^*$ that
$$
p^\alpha \cL[u_n](p) = -\lambda_n \langle \cL[u](p) , \phi_n \rangle_\rho -\int_{\pd \Omega} \cL[f](p,x) \psi_n(x) \dd \sigma(x).
$$
As a consequence we have
$$ 
\cL[u_n](p)=-( p^\alpha + \lambda_n )^{-1} \cL \left[\int_{\pd \Omega} f(\cdot,x) \psi_n(x) \dd \sigma(x) \right](p),\  p \in \R_+^*,
$$
for every $n \in \N^*$. This, \cite[Eq. (1.80)]{P} and the injectivity of the Laplace transform, yield \eqref{l1b}.
\end{proof}

We turn now to proving the main result of this section.

\begin{prop}
\label{pr-rf} 
Let $\alpha\in (0,1) \cup (1,2)$ and let $\rho$, $a$, and $V$, be as in Proposition \ref{pr-solution}. Pick $f \in \mathcal H_{\mathrm{in},\alpha,T_0}$, where $T_0 \in (0,T)$ is fixed, and let $u$ be the solution to \eqref{eq1}.
Then, for a.e. $x \in \pd \Omega$, we have
\bel{l2a}
a(x) \pd_\nu  u(T_0,x)=\int_0^{T_0} s^{\alpha-1} \left( \sum_{n=1}^\infty E_{\alpha,\alpha}(-\lambda_n s^\alpha) \left( \int_{\pd \Omega}  f(T_0-s,y) \psi_n(y) \dd\sigma(y) \right) \psi_n(x) \right) \dd s.
\ee
\end{prop}
\begin{proof} 
Let us first establish for each $s \in (0,T_0)$ that the series $\sum_{n=1}^{+\infty} \gamma_n(s) \varphi_n$, where
\bel{l2b}
\gamma_n(s):=-s^{\alpha-1} E_{\alpha,\alpha}(-\lambda_n s^\alpha) \left( \int_{\pd \Omega}  f(T_0-s,y) \psi_n(x) \dd\sigma(x) \right),\ n \in \N^*,
\ee
converges in $H^2(\Omega)$. Actually, since the domain of the operator $H$ is continuously embedded in $H^2(\Omega)$, it is enough to check that $\sum_{n=1}^{+\infty} \lambda_n \gamma_n(s) \varphi_n$ converges in $L^2_\rho(\Omega)$. This can be achieved with the help of \eqref{es2}. Indeed, in view of \eqref{l2b} we get through elementary computations that
\bel{l2c} 
| \gamma_n(s) | \leq c s^{-(1+\alpha)} \lambda_n^{-2} \left| \int_{\pd \Omega}  f(T_0-s,y) \psi_n(x) \dd\sigma(x) \right|,\ s \in (0,T_0),\  n \in \N^*,
\ee
where $c$ is the same as in \eqref{es2}.
This entails
\bea
\sum_{n=1}^{+\infty} \lambda_n^2 | \gamma_n(s) |^2 & \leq & C_\alpha^2 s^{-2(1+\alpha)} \left( \sum_{n=1}^{+\infty} \lambda_n^{-2} \left| \int_{\pd \Omega}  f(T_0-s,y) \psi_n(x) \dd\sigma(x) \right|^2 \right) \nonumber \\
& \leq & c^2 s^{-2(1+\alpha)} \| f(T_0-s,\cdot) \|_{H^{3 \slash 2}(\pd \Omega)}^2,\ s \in (0,T_0), \label{l2d}
\eea
upon applying Lemma \ref{lm-elliptic} with $g=f(T_0-s,\cdot)$, for some constant $c>0$, independent of $s$ and $f$.
As a consequence $\sum_{n=1}^{+\infty} \lambda_n \gamma_n(s) \varphi_n$ converges in $L^2_\rho(\Omega)$ for every $s \in (0,T_0)$, and hence $\sum_{n=1}^{+\infty} \gamma_n(s) \varphi_n$ converges in $H^2(\Omega)$. 

Next, since $\supp{f} \subset (0,T_0) \times \Gi$, it is apparent that $s \mapsto  s^{-(1+\alpha)} \| f(T_0-s,\cdot) \|_{H^{3 \slash 2}(\pd \Omega)} \in L^1(0,T_0)$, and similarly, we see from \eqref{l2d} that $s \mapsto \sum_{n=1}^N \lambda_n \gamma_n(s) \varphi_n \in L^1(0,T_0;L_\rho^2(\Omega))$ for every $N \in \N^*$, as we have
$$ \left\| \sum_{n=1}^N \lambda_n \gamma_n(s) \varphi_n \right\|_{\rho} = \left( \sum_{n=1}^N \lambda_n^2 | \gamma_n(s) |^2 \right)^{1 \slash 2} \leq C s^{-(1+\alpha)} \| f(T_0-s,\cdot) \|_{H^{3 \slash 2}(\pd \Omega)},\ s \in (0,T_0), $$
with $C>0$, independent of $s$ and $N$.
Therefore, $s \mapsto \sum_{n=1}^{+\infty} \gamma_n(s) \varphi_n \in L^1(0,T_0;H^2(\Omega))$ by \eqref{l2d} and Lebesgue dominated convergence theorem, and the identity
\bel{l2d2}
\sum_{n=1}^{+\infty} \left( \int_0^{T_0} \gamma_n(s) \dd s \right) \varphi_n = \int_0^{T_0} \left( \sum_{n=1}^{+\infty} \gamma_n(s) \varphi_n \right) \dd s,
\ee
holds in $H^2(\Omega)$. Recalling from Lemma \ref{l1} that $u(T_0,\cdot) = \sum_{n=1}^{+\infty} \left( \int_0^{T_0} \gamma_n(s) \dd s \right) \varphi_n$ in $L_\rho^2(\Omega)$, then by uniqueness of the limit, we end up getting from \eqref{l2d2}. Moreover, the identity
\bel{l2e}
u(T_0,\cdot) = \int_0^{T_0} \left( \sum_{n=1}^{+\infty} \gamma_n(s) \varphi_n \right) \dd s,
\ee
holds in $H^2(\Omega)$.

Finally, we obtain \eqref{l2a} from the continuity of the trace map $w \in H^2(\Omega) \mapsto \pd_{\nu} w_{| \pd \Omega} \in L^2(\pd \Omega)$ by mimicking all the steps of the derivation of \eqref{l2e}.
\end{proof}

\section{Proof of Theorem \ref{t2}}
\label{sec-proof}
The proof is divided into 4 steps.

\noindent {\it Step 1: Set up.} For $j=1,2$, let $H_j$ be the operator defined by \eqref{op-h} with $\rho=\rho_j$, $a=a_j$ and $V=V_j$, and let $\{ \lambda_{j,n}; n \in \N^* \}$ be the strictly increasing sequence of the eigenvalues of $H_j$. For each $n \in \N^*$, we denote by $m_{j,n} \in \N^*$ the algebraic multiplicity of the eigenvalue $\lambda_{j,n}$ and we introduce a family $\{ \varphi_{j,n,p};\ p=1,\ldots, m_{j,n} \}$ of eigenfunctions of $H_j$, which satisfy 
$$ H_j \varphi_{j,n,p} = \lambda_{j,n} \varphi_{j,n,p}, $$
and form an orthonormal basis in $L_{\rho_j}^2(\Omega)$ of the algebraic eigenspace of $H_j$ associated with $\lambda_{j,n}$ (i.e. the linear sub-space of $L_{\rho_j}^2(\Omega)$ spanned by $\{ \varphi_{j,n,p},\ p=1,\ldots, m_{j,n} \}$). Further, we put for a.e. $(x , y) \in \pd \Omega$,
$$ \Theta_{j,n}(x,y) := \sum_{p=1}^{m_{j,n}} \psi_{j,n,p}(x) \psi_{j,n,p}(y),\ \mbox{where}\ \psi_{j,n,p} := ( a \pd_\nu\phi_{j,n,p} )_{|\partial\Omega}. $$ 
Then, with reference to \eqref{def-DN} and Proposition \ref{pr-rf}, it holds true for every $f \in \Hin$ that
$$ \Lambda_{\rho_j,a_j,V_j} f = \int_0^{T_0} s^{\alpha-1} \left( \sum_{n=1}^{+\infty} E_{\alpha,\alpha}(-\lambda_{j,n} s^{\alpha}) \left( \int_{\partial\Omega} f(T_0-s,y)\Theta_{j,n}(\cdot,y) \dd\sigma(y) \right) \right) \dd s. $$
From this and the assumption $\Lambda_{\rho_1,a_1,V_1}=\Lambda_{\rho_2,a_2,V_2}$ then follows for a.e. $x \in \Go$, that
\bel{id1}
\int_0^{T_0} s^{\alpha-1} \left( \sum_{n=1}^{+\infty} \int_{\pd \Omega} \left( E_{\alpha,\alpha}(-\lambda_{1,n} s^\alpha) \Theta_{1,n}(x,y) - E_{\alpha,\alpha}(-\lambda_{2,n} s^\alpha) \Theta_{2,n}(x,y) \right) f(T_0-s,y) \dd \sigma(y) \right) \dd s=0.
\ee
In view of the integrand appearing in the left hand side of \eqref{id1}, we introduce for every $h \in H^{3 \slash 2}(\pd \Omega)$ such that $\supp{h} \subset \Gi$, the following function
\bel{l3a}
F_{j,h}(z,x) := \sum_{n=1}^{+\infty} E_{\alpha,\alpha}(-\lambda_{j,n} z) \left( \int_{\Gi} \Theta_{j,n}(x,y) h(y) \dd \sigma(y) \right),\ z \in \C,\ x \in \Go.
\ee
Then, the identity \eqref{id1} being valid for every $f \in \Hin$, we find upon taking $f(t,x)=g(t)h(x)$, for $(t,x) \in (0,T_0) \times \Gi$, where $g$ is arbitrary in $C_0^\infty(0,T_0)$ and $h \in H^{3 \slash 2}(\pd \Omega)$ is, as above, supported in $\Gi$, that
\bel{t2a} 
F_{1,h}(s^\alpha,x) = F_{2,h}(s^\alpha,x),\ s \in (0,T_0),\ x \in \Go. 
\ee


\noindent {\it Step 2: Analytic continuation.} We start by establishing the following technical result.

\begin{lem}
\label{l3} 
Fix $\theta_0 \in ( \pi \alpha \slash 2, \pi \alpha)$ and pick $h \in H^{3 \slash 2}(\pd \Omega)$ satisfying $\supp{h} \subset \Gi$. 
Then, both $L^2(\Go)$-valued functions $z \mapsto F_{j,h}(z,\cdot)$, $j=1,2$, defined in \eqref{l3a}, are holomorphic in the sub-domain $\cD_{\theta_0} :=\left\{ z \in \C \setminus \{ 0 \};\ | \mbox{arg}(z) | < \pi - \theta_0 \right\}$.
\end{lem}
\begin{proof} 
Let $j$ be either 1 or 2. Bearing in mind that 
$$
\int_{\Gi} \Theta_{j,n}(\cdot,y) h(y) \dd \sigma(y)
= \sum_{p=1}^{m_{j,n}} \left( \int_{\Gi} h(y) \psi_{j,n,p}(y) \dd \sigma(y) \right) \psi_{j,n,p},\ n \in \N^*,
$$
we see upon arguing as in the derivation of Proposition \ref{pr-rf}, that it is enough to show that the $L_\rho^2(\Omega)$-valued function 
\bel{defGj}
G_j(z):= \sum_{n=1}^{+\infty} E_{\alpha,\alpha}(-\lambda_{j,n} z) \left( \int_{\Gi} h(y) \psi_{j,n}(y) \dd \sigma(y) \right) \lambda_{j,n} \varphi_n,
\ee
is holomorphic in $\cD_{\theta_0}$. 

Further, as $\alpha \in (0,2)$ and $\abs{\mbox{arg}(-z)} \in [\theta_0,\pi]$, we invoke \eqref{es1} and get some positive constant $c$ such that
$$
\abs{E_{\alpha,\alpha}(-\lambda_{j,n} z)} \leq \frac{c}{1+\abs{\lambda_{j,n} z}^2},\ z \in \cD_{\theta_0},\ n \in \N^*.
$$
As a consequence we have
\bel{n1} 
\lambda_{j,n} |E_{\alpha,\alpha}(-\lambda_{j,n} z)| \left| \int_{\Gi} h(y) \psi_{j,n}(y) \dd \sigma(y) \right| 
\leq c | z |^{-2} \lambda_{j,n}^{-1} \left| \int_{\pd \Omega} h(y) \psi_{j,n}(y) \dd \sigma(y) \right|.
\ee
Let $K$ be a compact subset of $\cD_{\theta_0}$. Due to Lemma \ref{lm-elliptic}, \eqref{n1} yields that the series appearing in the right hand side of \eqref{defGj} converges in $L^2_{\rho_j}(\Omega)$, uniformly in $z \in K$. Next, the mapping 
$z \mapsto  E_{\alpha,\alpha}(-\lambda_{j,n} z) \left( \int_{\Gi} h(y) \psi_{j,n}(y) \dd \sigma(y) \right) \lambda_{j,n} \varphi_n$ being holomorphic in $K$ for each $n \in \N^*$, since the Mittag-Leffler function $E_{\alpha,\alpha}$ is holomorphic in $\C$ from the very definition \eqref{def-MLf}, we get that $G_j$ is holomorphic in $K$ as well.
This entails that $G_j$ is analytic in $\cD_{\theta_0}$ since $K$ is arbitrary in $\cD_{\theta_0}$.
\end{proof}

Fix $\theta_0 \in ( \pi \alpha \slash 2, \pi \alpha)$. Since $F_{1,h}(z,x) = F_{2,h}(z,x)$ for a.e. $x \in \Go$ and all $z \in (0,T_0^\alpha)$, according to \eqref{t2a}, then the same is true for $z \in \cD_{\theta_0}$, by Lemma \ref{l3} in virtue of the unique continuation principle for analytic functions. This yields 
\bel{tt2a} 
F_{1,h}(t,x) = F_{2,h}(t,x),\ x \in \Go,\ t \in (0,+\infty).
\ee
Having seen this, we turn now to proving that the identity \eqref{tt2a} yields
\bel{t2b}
\lambda_{1,n}=\lambda_{2,n}\ \mbox{and}\ \Theta_{1,n}(x,y)=\Theta_{2,n}(x,y),\ n \in \N^*,\ (x,y) \in \Go \times \Gi.
\ee
This can be done upon computing the Laplace transform of both sides of \eqref{tt2a}. 

\noindent {\it Step 3: Laplace transform.} We fix $N \in \N^*$, $t \in \R_+^*$ and recall that the $L^2(\pd \Omega)$-norm of 
$$ \sum_{n=1}^N E_{\alpha,\alpha}(-\lambda_{j,n} t^\alpha) \left( \int_{\Gi}\Theta_{j,n}(\cdot,y) h(y) \dd \sigma(y) \right)
= \sum_{n=1}^N E_{\alpha,\alpha}(-\lambda_{j,n} t^\alpha) \left( \sum_{p=1}^{m_{j,n}} \left(  \int_{\Gi} h(y) \psi_{j,n,p}(y) \dd \sigma(y) \right) \psi_{j,n,p} \right), $$ 
is upper bounded (up to some positive multiplicative constant that depends only on $\Omega$) by the $L^2_{\rho_j}(\Omega)$-norm of $\sum_{n=1}^N E_{\alpha,\alpha}(-\lambda_{j,n} t^\alpha) \left( \sum_{p=1}^{m_{j,n}} \left( \int_{\Gi} h(y) \psi_{j,n,p}(y) \dd\sigma(y) \right) \lambda_{j,n} \varphi_{j,n,p} \right)$. Hence we find that
\beas
& & \norm{\sum_{n=1}^N t^{\alpha+1} E_{\alpha,\alpha}(-\lambda_{j,n} t^\alpha) \left( \int_{\Gi}\Theta_{j,n}(\cdot,y) h(y) \dd \sigma(y) \right)}_{L^2(\Go)} \\
& \leq & C t^{\alpha+1} \left( \sum_{n=1}^N \lambda_{j,n}^2 E_{\alpha,\alpha}(-\lambda_{j,n} t^\alpha)^2 \left( \sum_{p=1}^{m_{j,n}} \left| \int_{\pd \Omega} h(y) \psi_{j,n,p}(y) \dd \sigma(y) \right|^2 \right) \right)^{1 \slash 2} \\
& \leq & C t^{1-\alpha} \left( \sum_{n=1}^N  \lambda_{j,n}^{-2} \left( \sum_{p=1}^{m_{j,n}} \left| \int_{\pd \Omega} h(y) \psi_{j,n,p}(y) \dd \sigma(y) \right|^2 \right) \right)^{1 \slash 2},
\eeas
according to \eqref{es2}, the constant $C>0$ depending neither on $N$, nor on $t$. Therefore, we have
$$ \norm{\sum_{n=1}^N t^{\alpha+1} E_{\alpha,\alpha}(-\lambda_{j,n} t^\alpha) \left( \int_{\Gi}\Theta_{j,n}(\cdot,y) h(y) \dd \sigma(y) \right)}_{L^2(\Go)} \leq C t^{1-\alpha} \| h \|_{H^{3 \slash 2}(\pd \Omega)}, $$
by Lemma \ref{lm-elliptic}, and the Lebesgue dominated convergence theorem for $L^2(\Go)$-valued functions yields
\bea
& & \cL \left[\sum_{n=1}^{+\infty}  t^{\alpha+1} E_{\alpha,\alpha}(-\lambda_{j,n} t^\alpha) \left( \int_{\Gi} \Theta_{j,n}(x,y) h(y) \dd \sigma(y) \right) \right](p) \nonumber \\
& =& \sum_{n=1}^{+\infty} \left( \int_{\Gi} \Theta_{j,n}(x,y) h(y) \dd \sigma(y) \right) \cL \left[ t^{\alpha+1} E_{\alpha,\alpha}(-\lambda_{j,n} t^\alpha) \right](p),\ x \in \Go,\ p \in \R_+^*. \label{t2s} 
\eea
We are thus left with the task of computing the Laplace transform of $t \mapsto t^{\alpha+1} E_{\alpha,\alpha}(-\lambda_{j,n} t^\alpha)$ on $\R_+^*$. We find by standard computations that \cite[Eq. (1.80)]{P} implies
\bel{lap}
\cL \left[  t^{\alpha+1} E_{\alpha,\alpha}(-\lambda_{j,n} t^\alpha) \right](p) =  \frac{\dd}{\dd p}^2 \left( \cL \left[  t^{\alpha-1} E_{\alpha,\alpha}(-\lambda_{j,n} t^\alpha) \right](p) \right) = p^{\alpha-2}\left(\frac{2\alpha^2p^\alpha-\alpha(\alpha-1)(p^\alpha+\lambda_{j,n})}{(p^\alpha+\lambda_{j,n})^3}\right),
\ee
for all $p>\lambda_{j,n}^{1 \slash \alpha}$.
Further, since $p \mapsto \cL \left[  t^{\alpha+1} E_{\alpha,\alpha}(-\lambda_{j,n} t^\alpha)\right](p)$ is an analytic function of  $p \in \{z \in \C;\ \Re{z}>0\}$, then \eqref{lap} holds for every $p \in \R_+^*$, and we have
\bel{t2c} \begin{array}{l}
\sum_{n=1}^{+\infty} \frac{(2\alpha^2p^\alpha-\alpha(\alpha-1)(p^\alpha+\lambda_{1,n}))\int_{\Gi} \Theta_{1,n}(x,y) h(y) \dd \sigma(y)}{(p^\alpha+\lambda_{1,n})^3} \\
=
\sum_{n=1}^{+\infty} \frac{(2\alpha^2p^\alpha-\alpha(\alpha-1)(p^\alpha+\lambda_{2,n}))\int_{\Gi} \Theta_{2,n}(x,y) h(y) \dd \sigma(y)}{(p^\alpha+\lambda_{2,n})^3},\ p \in \R_+^*,\ x \in \Go,\end{array}
\ee
by \eqref{l3a}, \eqref{tt2a} and  \eqref{t2s}.

\noindent {\it Step 4: End of the proof.} Consider $\cO :=\C \setminus \{-\lambda_{j,n};\ j=1,2,\ n \in \N^* \}$ and note that, since, for $j=1,2$,   $(\lambda_{j,n})_{n\geq1}$ is a strictly increasing and unbounded sequence, the set $\cO$ is connected. Let $K$ be a compact subset of $\cO :=\C \setminus \{-\lambda_{j,n};\ j=1,2,\ n \in \N^* \}$. Arguing as in the derivation of Lemma \ref{l3}, we see that, for $j=1,2$, the serie
$$\sum_{n=1}^{+\infty} \left(\frac{2\alpha^2z-\alpha(\alpha-1)(z+\lambda_{j,n})}{(z+\lambda_{j,n})^3}\right)  \left( \int_{\Gi}\Theta_{j,n}(x,y) h(y) \dd \sigma(y) \right),$$
converges uniformly with respect to $z \in K$. Thus, since $K$ is arbitrary in $\cO$, the function
$$ z \mapsto  \sum_{n=1}^{+\infty} \left(\frac{2\alpha^2z-\alpha(\alpha-1)(z+\lambda_{j,n})}{(z+\lambda_{j,n})^3}\right) \left( \int_{\Gi} \Theta_{j,n}(\cdot,y) h(y) \dd \sigma(y) \right), $$
 is analytic in $\cO$  and  we deduce from \eqref{t2c},
that 
\bel{t2d}\begin{array}{l}
\sum_{n=1}^{+\infty} \left(\frac{2\alpha^2z-\alpha(\alpha-1)(z+\lambda_{1,n})}{(z+\lambda_{1,n})^3}\right) \left( \int_{\Gi}\Theta_{1,n}(x,y) h(y) \dd \sigma(y) \right)\\
 =
\sum_{n=1}^{+\infty} \left(\frac{2\alpha^2z-\alpha(\alpha-1)(z+\lambda_{2,n})}{(z+\lambda_{2,n})^3}\right) \left( \int_{\Gi}\Theta_{2,n}(x,y) h(y) \dd \sigma(y) \right),\ z \in \cO,\ x \in \Go,\end{array}
\ee
by the unique continuation principle for analytic functions. 
Now, putting $\lambda_{*1}:=\underset{(j,n) \in \{ 1,2 \}\times \N^*}{\min} \lambda_{j,n}$, multiplying both sides of \eqref{t2d} by $(z+\lambda_{*1})^3$, and sending $z$ to $(-\lambda_{*1})$, we obtain that
$$ \lambda_{1,1}=\lambda_{*1}=\lambda_{2,1}\ \mbox{and}\
\int_{\Gi} \Theta_{1,1}(x,y) h(y) \dd \sigma(y) = \int_{\Gi} \Theta_{2,1}(\cdot,y) h(y) \dd \sigma(y)\ \mbox{for\ a.e.}\ x \in \Go. $$
Similarly, by induction on $n \in \N^*$, we find that
$$ \lambda_{1,n}=\lambda_{2,n}\ \mbox{and}\
\int_{\Gi} \Theta_{1,n}(x,y) h(y) \dd \sigma(y) = \int_{\Gi} \Theta_{2,n}(\cdot,y) h(y) \dd \sigma(y)\ \mbox{for\ a.e.}\ x \in \Go. $$
for any function $h \in H^{3 \slash 2}(\pd \Omega)$ supported in $\Gi$, which yields \eqref{t2b}. 
Finally, since $\Gi \cup \Go=\pd \Omega$ and $\Gi \cap \Go\neq\emptyset$, we end up getting that $\mbox{BSD}(\rho_1,a_1,V_1)=\mbox{BSD}(\rho_2,a_2,V_2)$, up to an appropriate choice of the functions $\{ \varphi_{1,n},\ n \in \N^* \}$, from \eqref{t2b} and the end of the proof of Theorem 1.1 in \cite{CK1} (see \cite[pages 975-976]{CK1}).

\section{Results on Riemannian manifolds}
\label{sec-Riemann}
\def\p{\partial}

In this section we prove Theorem \ref{t5}. Then, we establish in the particular case where $\mu=1$ and $q=0$, upon assuming a spectral Hassell-Tao type inequality (see \eqref{Hassell_Tao} below) that the result of Theorem \ref{t5} remains valid when $\overline{\Gi} \cap \overline{\Go} = \emptyset$.

Nevertheless, in the first step of the analysis we assume a slightly more restrictive assumption, i.e. that $\Gi \cap \Go \ne \emptyset$, than the one required by Theorem \ref{t5}.
We consider the weighted measure $\mu dx$, where $dx$ is the Riemannian volume measure, to define the space $L^2(M)$,
since $\Delta_{g,\mu}$ is symmetric with respect to the resulting inner product.
Let us introduce the elliptic operator $A$ acting on $L^2(M)$ with domain 
$D(A)= H^1_0(M)\cap H^2(M)$ 
defined by
\begin{align}
\label{def_op_A}
A h=-\Delta_{g,\mu}h + q h,\quad h\in D(A).
\end{align}
By a compact resolvent argument  we know that the spectrum of $A$ consists of a non-decreasing sequence of eigenvalues $(\lambda_{n})_{n\geq 1}$
and we can introduce the associated  Hilbertian basis of eigenfunctions $(\phi_{n})_{n\geq1}$. 
We define the boundary spectral data as
$$ 
\mathrm{BSD}(M,g,\mu,q;\Gamma) :=\{ (\lambda_n, \psi_{n\vert \Gamma} );\ n\geq1\},
$$
where $\Gamma \subset \pd M$ is open and $\psi_n = \pd_\nu \phi_n$. In view of these BSD, it is easy to see that the results of Sections \ref{sec-set}, \ref{sec-direct} and \ref{sec-proof} remain valid in the framework of Theorem \ref{t5}. In particular, we may repeat the proof of Theorem \ref{t2} 
in the present context to obtain:
\begin{thm}
\label{t3} 
Let $(M_k,g_k)$, $k=1,2$, be two compact and smooth connected  Riemannian manifolds of dimension $d\geq2$ with the same boundary. Let $\mu_k, q_k \in C^\infty(M_k)$
satisfy $\mu_k(x) > 0$ and $q_k(x) \ge 0$
for all $x \in M_k$, $k=1,2$.
Let $\Gi, \Go \subset \pd M_1$
be relatively open
and suppose that $\Go\cap \Gi$ is nonempty.
Suppose, moreover, that $g_1 = g_2$, $\mu_1 = \mu_2=1$
and $\pd_\nu \mu_1 = \pd_\nu \mu_2 = 0$ on $\pd M_1$.
Then, the condition $\Lambda_{M_1,g_1,\mu_1,q_1}= \Lambda_{M_2,g_2,\mu_2,q_2}$ implies that, up to an appropriate choice of the eigenfunctions of the operator $A_1$, we have 
\begin{align}
\label{Mg_BSD_equal}
\mathrm{BSD}(M_1,g_1,\mu_1,q_1;\Go\cup \Gi) 
=
\mathrm{BSD}(M_2,g_2,\mu_2,q_2;\Go\cup \Gi).
\end{align}
\end{thm}

It is well-known that (\ref{Mg_BSD_equal})
implies that $(M_k, g_k)$, $k=1,2$, are isometric,
and that $(\mu_k, q_k)$, $k=1,2$, are on the same orbit of the group of gauge transformations, and
we refer to \cite{Katchalov2001} for a detailed proof.
To our knowledge, all the proofs of this result are based on the Boundary Control method. 
The Boundary Control method was introduced by Belishev in 
\cite{Belishev1987}
where he solved the inverse boundary value problem for the isotropic wave equation, that is, the equation (\ref{eq1})
with $\alpha = 2$, $a = 1$ and $q=0$.
The method was generalized to geometric context
by Belishev and Kurylev \cite{Belishev1992},
and the inverse boundary spectral problem with partial data 
as in \eqref{Mg_BSD_equal} was solved by Katchalov and Kurylev
\cite{Katchalov1998}.
We mention that a reduction similar to Theorem \ref{t3}
was shown in \cite{Katchalov2004} in the case of the heat equation, that is, the equation \eqref{eqM1} with $\alpha = 1$.

Let us now consider two further generalization where the assumption that $\Go\cap \Gi$ is nonempty
is weakened. These generalizations are based on the observation that the proof of Theorem \ref{t2}
gives:

\begin{thm}
\label{t4} 
Let $(M_k,g_k)$, $k=1,2$, be two compact and smooth connected  Riemannian manifolds of dimension $d\geq2$ with the same boundary. Let $\mu_k, q_k \in C^\infty(M_k)$
satisfy $\mu_k(x) > 0$ and $q_k(x) \ge 0$
for all $x \in M_k$, $k=1,2$.
Let $\Gi, \Go \subset \pd M_1$
be relatively open,
and suppose that $g_1 = g_2$, $\mu_1 = \mu_2 = 1$
and $\pd_\nu \mu_1 = \pd_\nu \mu_2 = 0$ on $\pd M_1$.
Then, the condition $\Lambda_{M_1,g_1,\mu_1,q_1}= \Lambda_{M_2,g_2,\mu_2,q_2}$ implies that 
\begin{align}
\label{Mg_Theta_equal}
\lambda_{1,n}=\lambda_{2,n}\ \mbox{and}\ \Theta_{1,n}(x,y)=\Theta_{2,n}(x,y),\ n \in \N^*,\ (x,y) \in \Go \times \Gi,
\end{align}
where, as before, $\{\lambda_{k,n}; n \in \N^*\}$ is the strictly increasing sequence of the Dirichlet eigenvalues of $A_k$ and 
$$
\Theta_{k,n}(x,y) := \sum_{p=1}^{m_{k,n}} \psi_{k,n,p}(x) \psi_{k,n,p}(y),\ \mbox{where}\ \psi_{k,n,p} :=\sum_{i,j=1}^d g_k^{ij}\nu_i\partial_{x_j} \phi_{k,n,p}.
$$
Here the eigenfunctions $\phi_{k,n,p}$ are chosen so that $\phi_{k,n,p}$, $p=1,\dots,m_{k,n}$, form an orthonormal basis of the eigenspace associated with $\lambda_{k,n}$.
\end{thm}

If $\overline{\Gi} \cap \overline{\Go} \ne \emptyset$,
then
the equation (\ref{Mg_Theta_equal})
implies that the boundary spectral data 
\begin{align}
\label{bsd_ge}
\mathrm{BSD}(M_k,g_k,\mu_k,q_k;\Go), \quad k=1,2,
\end{align}
are gauge equivalent in the sense that
there is a constant $\kappa > 0$
such that up to an appropriate choice of the eigenfunctions of the operator $A_1$, we have 
$\psi_{1,n,p} = \kappa \psi_{2,n,p}$ on $\Go$
for all $n$ and $p$, see
\cite[Theorem 4]{Lassas2010}.
In \cite{Lassas2010}
the authors considered
only operators $A$ of the form (\ref{def_op_A})
with $\mu = 1$ identically. We can actually reduce to this case upon taking $\kappa = \sqrt{\mu}$ in \eqref{gauge_mu_q}.
The gauge equivalence of the boundary spectral data
(\ref{bsd_ge})
implies that $(M_k, g_k)$, $k=1,2$, are isometric
and that $(\mu_k, q_k)$, $k=1,2$, are on the same orbit of the group of gauge transformations,
as can be seen by combining the proofs of 
\cite[Theorems 4.33 and 3.37]{Katchalov2001}. This proves Theorem \ref{t5}.

Finally, let us consider the case where $\Gi$ and $\Go$
are allowed to be far apart. 
Following \cite{LO} we assume that
$\mu_k = 1$ and $q_k = 0$ identically, and that
 both $(M_k, g_k)$, $k=1,2$, satisfy the spectral inequality
\begin{align}
\label{Hassell_Tao}
\lambda_{k,n} \le C \|\psi_{k,n,p}\|_{L^2(\Gi)}^2,
\end{align}
where the constant $C > 0$ is independent of $n$ and $p$.
Hassell and Tao \cite{Hassell2002} showed that all non-trapping Riemannian manifolds $(M_k, g_k)$ satisfy (\ref{Hassell_Tao}) when $\Gi$ is replaced by $\pd M_k$. Moreover,
(\ref{Hassell_Tao}) follows from (and is strictly weaker than) the geometric control condition by Bardos, Lebeau and Rauch \cite{Bardos1992},
see \cite{LO}.
We will now give a reduction to the result in \cite{LO}.
Let us denote by $L_{M,g}$ the hyperbolic DN map associated to the Riemannian manifold $(M,g)$ and restricted to $\Gi \times \Go$, that is, 
$$
L_{M,g}f = \pd_\nu u|_{(0,\infty) \times \Go},
\quad f \in C_0^\infty((0,\infty) \times \Gi),
$$
where $u$ is the solution of (\ref{eqM1})
with $\mu = 1$, $q=0$ identically and $\alpha = 2$. The map $L_{M_k,g_k}$ has the representation
$$
L_{M_k,g_k}f(t,x) = \sum_{n \in \N^*} \int_0^t \int_{\Gi}
f(s,y) \frac{\sin(\sqrt{\lambda_{k,n}} (t-s))}{\sqrt{\lambda_{k,n}}}
\Theta_{k,n}(x,y) dy ds,
$$
where $dy$ is the Riemannian surface measure on $\pd M_1$, 
see e.g. \cite[Lemma 3.6]{Katchalov2001}.
Hence (\ref{Mg_Theta_equal})
implies that $L_{M_1,g_1} = L_{M_2,g_2}$,
and therefore $(M_k, g_k)$, $k=1,2$, are isometric \cite{LO}. We have shown:

\begin{thm}
\label{t6}
Let $(M_k,g_k)$, $k=1,2$, be two compact and smooth connected  Riemannian manifolds of dimension $d\geq2$ with the same boundary.
Let $\Gi, \Go \subset \pd M_1$
be relatively open, and suppose that $g_1 = g_2$ on $\pd M_1$.
Suppose, moreover, that both $(M_k, g_k)$, $k=1,2$, satisfy the spectral inequality (\ref{Hassell_Tao}).
Then,  the condition $\Lambda_{M_1,g_1, 1, 0}= \Lambda_{M_2,g_2, 1, 0}$ implies that $(M_k, g_k)$, $k=1,2$, are isometric.
\end{thm}

We do not know if Theorem \ref{t6}
holds for operators with varying $\mu$ and $q$, see the discussion in 
\cite[pp. 7-8]{Lassas2014}.

\section*{Acknowledgments}
The first author would like to thank Lorenzo Brasco for his remarks and  fruitful discussions. The fourth author is partially supported by Grant-in-Aid for Scientific
Research (S) 15H05740 of Japan Society for the Promotion of Science.

\vspace*{1cm}
\noindent {\sc Yavar Kian}, Aix-Marseille Universit\'e, CNRS, CPT UMR 7332, 13288 Marseille, and Universit\'e de Toulon, CNRS, CPT UMR 7332, 83957 La Garde, France.\\
E-mail: {\tt yavar.kian@univ-amu.fr}. \vspace*{.5cm} \\
\noindent {\sc Lauri Oksanen}, Department of Mathematics, University College London, Gower Street, London, WC1E 6BT, UK.\\
E-mail: {\tt l.oksanen@ucl.ac.uk}. \vspace*{.5cm}\\
\noindent {\sc Eric Soccorsi}, Aix-Marseille Universit\'e, CNRS, CPT UMR 7332, 13288 Marseille, and Universit\'e de Toulon, CNRS, CPT UMR 7332, 83957 La Garde, France.\\
E-mail: {\tt eric.soccorsi@univ-amu.fr}. \vspace*{.5cm}\\
\noindent {\sc Masahiro Yamamoto}, Department of Mathematical Sciences, The University of Tokyo 3-8-1, Komaba, Meguro, Tokyo 153, Japan.\\
E-mail: {\tt myama@ms.u-tokyo.ac.jp}.

\end{document}